\documentclass[11pt]{article}

\usepackage{amsmath, amsthm, 		amssymb, enumerate,
	mathtools, tikz, xcolor}
\usepackage[hidelinks=true]{hyperref}
\usepackage[margin=1.25in]{geometry}

\newtheorem{theorem}{Theorem}[section]
\newtheorem{definition}{Definition}
\newtheorem{lemma}{Lemma}[section]
\newtheorem{corollary}{Corollary}[section]
\newtheorem{conjecture}{Conjecture}[section]

\newtheorem{observation}{Observation}
\theoremstyle{definition}
\newtheorem{claim}{Claim}
\theoremstyle{definition}

\title{The Extremal Function and Colin de Verdi\`{e}re Graph Parameter}
\author{
\textbf{Rose McCarty}\footnote{\href{mailto:rmccarty3@gatech.edu}{\texttt{rmccarty3@gatech.edu}}. Partially supported by NSF under Grant No. DMS-1202640.}\\
School of Mathematics\\
Georgia Institute of Technology\\
Atlanta, GA 30332-0160, USA
}
\date{}

\begin{document}

\maketitle

\begin{abstract}
We study the maximum number of edges in an $n$ vertex graph with Colin de Verdi\`{e}re parameter no more than $t$. We conjecture that for every integer $t$, if $G$ is a graph with at least $t$ vertices and Colin de Verdi\`{e}re parameter at most $t$, then $|E(G)| \leq t|V(G)|-\binom{t+1}{2}$. We observe a relation to the graph complement conjecture for the Colin de Verdi\`{e}re parameter and prove the conjectured edge upper bound for graphs $G$ such that either $\mu(G) \leq 7$, or $\mu(G) \geq |V(G)|-6$, or the complement of $G$ is chordal, or $G$ is chordal.
\end{abstract}

\section{Introduction}
We consider only finite, simple graphs without loops. Let $\mu(G)$ denote the Colin de Verdi\`{e}re parameter of a graph $G$ introduced in \cite{cdvIntro} (cf. \cite{cdvIntroTranslated}). We give a formal definition of $\mu(G)$ in Section 2. The Colin de Verdi\`{e}re parameter is minor-monotone; that is, if $H$ is a minor of $G$, then $\mu(H) \leq \mu(G)$. Particular interest in this parameter stems from the following characterizations:  

\begin{theorem} For every graph $G$:
\begin{enumerate}
\itemsep0em
\item $\mu(G) \leq 1$ if and only if $G$ is a subgraph of a path.
\item $\mu(G) \leq 2$ if and only if $G$ is outerplanar.
\item $\mu(G) \leq 3$ if and only if $G$ is planar.
\item $\mu(G) \leq 4$ if and only if $G$ is linklessly embeddable.
\end{enumerate}
\end{theorem}

Items 1, 2, and 3 were shown by Colin de Verdi\`{e}re in \cite{cdvIntro}. Robertson, Seymour, and Thomas noted in \cite{linklessSurvey} that $\mu(G) \leq 4$ implies that $G$ has a linkless embedding due to their theorem that the Petersen family is the forbidden minor family for linkless embeddings \cite{linklessMinors}. The other direction for 4 is due to Lov\'{a}sz and Schrijver \cite{BorsukThm}. See the survey of van der Holst, Lov\'{a}sz, and Schrijver for a thorough introduction to the parameter \cite{survey}.

There is also a relation between the Colin de Verdi\`{e}re parameter and Hadwiger's conjecture that for every non-negative integer $t$, every graph with no $K_{t+1}$ minor is $t$-colorable. Let $\chi(G)$ denote the chromatic number of a graph $G$ and let $h(G)$ denote the Hadwiger number of $G$. That is, $h(G)$ is the largest integer so that $G$ has the complete graph $K_{h(G)}$ as a minor. Then $\mu(K_{h(G)}) = h(G)-1$, and so $\mu(G) \geq \mu(K_{h(G)}) = h(G)-1$ \cite{survey}. So if Hadwiger's conjecture is true, then for every graph $G$, $\chi(G) \leq \mu(G)+1$. Colin de Verdi\`{e}re conjectured that every graph satisfies $\chi(G) \leq \mu(G)+1$ in \cite{cdvIntro}. For graphs with $\mu(G) \leq 3$, this statement is exactly the 4-Color Theorem \cite{Appel}, \cite{RobertsonColour}.

One way to look for evidence for Hadwiger's conjecture is through considerations of average degree. In particular Mader showed that for every family of graphs $\mathcal{F}$, there is an integer $c$ so that if $G$ is a graph with no graph in $\mathcal{F}$ as a minor, then $|E(G)| \leq c|V(G)|$ \cite{Mader}. It follows by induction on the number of vertices that every graph $G$ with no graph in $\mathcal{F}$ as a minor is $2c+1$-colorable. In fact Mader showed that:

\begin{theorem} \cite{MaderEdgeBound}
\label{hadwigerEdgeBound}
For $t \leq 5$, if $G$ is a graph with $h(G) \leq t+1$ and $|V(G)| \geq t$, then $|E(G)| \leq t|V(G)|-\binom{t+1}{2}$.
\end{theorem}

However asymptotically, as noted by Kostochka \cite{avgDegHadwiger} and Thomason \cite{Thomason}, based on Bollob\'{a}s et at. \cite{BollobasEtAl}:

\begin{theorem} \cite{avgDegHadwiger}, \cite{Thomason}
There exists a constant $c \in \mathbb{R}^+$ such that for every positive integer $t$ there exists a graph $G$ with $h(G) \leq t+1$ and $|E(G)| > ct\sqrt{\log{t}}|V(G)|$. 
\end{theorem}

Furthermore, Kostochka showed that asymptotically in $t$ the same is an upper bound \cite{avgDegHadwiger}. This gives the best known bound on Hadwiger's conjecture, that graphs $G$ with no $K_t$ minor have $\chi(G) \leq \mathcal{O}(t\sqrt{\log{t}})$. We conjecture that an analog of Theorem \ref{hadwigerEdgeBound} holds instead for the Colin de Verdi\`{e}re parameter:

\begin{conjecture}
\label{mainConjecture}
For every integer $t$, if $G$ is a graph with $\mu(G) \leq t$ and $|V(G)| \geq t$, then $|E(G)| \leq t|V(G)|-\binom{t+1}{2}$.
\end{conjecture}

Nevo asked if this is true and showed that his Conjecture 1.5 in \cite{Nevo} implies Conjecture \ref{mainConjecture}. Tait also asked this question as Problem 1 in \cite{Tait} in relation to studying graphs with maximum spectral radius of their adjacency matrix, subject to having Colin de Verdi\`{e}re parameter at most $t$. We also observe that there is a relation between Conjecture \ref{mainConjecture} and the graph complement conjecture for the Colin de Verdi\`{e}re parameter. Let $\overline{G}$ denote the complement of $G$. The graph complement conjecture for the Colin de Verdi\`{e}re parameter is as follows:

\begin{conjecture}
\label{graphComplement}
For every graph $G$, $\mu(G)+\mu(\overline{G}) \geq |V(G)|-2$.
\end{conjecture}

This conjecture was introduced by Kotlov, Lov\'{a}sz, and Vempala, who showed that the conjecture is true if $G$ is planar \cite{sphereRep}. Their result is used in this paper and will be stated formally in Section 4. Conjecture \ref{graphComplement} is also an instance of a Nordhaus-Gaddum sum problem. See the recent paper by Hogben for a survey of Nordhaus-Gaddum problems for the Colin de Verdi\`{e}re and related parameters, including Conjecture \ref{graphComplement} \cite{Nordhaus-Gaddum}. We observe that:

\begin{observation}
\label{Obs:impliesWeakGCC}
If there exists a constant $c \in \mathbb{R^+}$ so that for every graph $G$, $|E(G)| \leq c\mu(G)|V(G)|$, then there exists a constant $p \in \mathbb{R}^+$ so that for every graph $G$, $\mu(G)+\mu(\overline{G}) \geq p|V(G)|$.
\end{observation}

This follows from noting that we would have $c\mu(G)|V(G)|+c\mu(\overline{G})|V(G)| \geq |E(G)|+|E(\overline{G})| = \binom{|V(G)|}{2}$. So our main Conjecture \ref{mainConjecture} would imply an asymptotic version of the graph complement conjecture for the Colin de Verdi\`{e}re parameter. This weaker version is currently not known. In the other direction we will show in Section 2 that:

\begin{observation}
\label{Obs:GCCimplies}
If for every graph $G$, $\mu(G)+\mu(\overline{G}) \geq |V(G)|-2$, then every graph $G$ has $|E(G)| \leq (\mu(G)+1)|V(G)| - \binom{\mu(G)+2}{2}$.
\end{observation}

Then in particular the graph complement conjecture for Colin de Verdi\`{e}re parameter would imply that all graphs $G$ are $2\mu(G)+2$-colorable. We will also show in Section 2 that:

\begin{observation}
\label{Obs:tight}
Let $H$ be any edge-maximal planar graph on at least 4 vertices and let $t \geq 3$ be an integer. Then let $G$ denote the join of $H$ and $K_{t-3}$. That is, $V(G)$ is the disjoint union of $V(H)$ and $V(K_{t-3})$, where the induced subgraph of $G$ with vertex set $V(H)$ is equal to $H$, the induced subgraph of $G$ with vertex set $V(K_{t-3})$ is complete, and for all vertices $u \in V(H)$ and $v \in V(K_{t-3})$, $uv \in E(G)$. Then $\mu(G) = t$ and $|E(G)| = t|V(G)|-\binom{t+1}{2}$.
\end{observation}

So for every positive integer $t$, Conjecture \ref{mainConjecture} is tight for infinitely many graphs. We say a graph $G$ is \textit{chordal} if for every cycle $C$ of $G$ of length greater than 3, the induced subgraph of $G$ with vertex set $V(C)$ has some edge that is not in $E(C)$. The main result we prove is Theorem \ref{mainThm}:

\begin{theorem}
\label{mainThm}
Suppose $G$ is a graph such that either:
\begin{itemize}
\itemsep0em
\item $G$ is chordal, or
\item $\overline{G}$ is chordal, or
\item $\mu(G) \leq 7$, or 
\item $\mu(G) \geq |V(G)|-6$.
\end{itemize}
Then $|E(G)| \leq \mu(G)|V(G)|-\binom{\mu(G)+1}{2}$.
\end{theorem}

Note that it is equivalent to say that for such graphs, for every integer $t$ with $\mu(G) \leq t \leq n$, $|E(G)| \leq t|V(G)|-\binom{t+1}{2}$. We also note that the analog of Theorem \ref{mainThm} for the Hadwiger number is false. For $n_1, n_2, \ldots, n_k \in \mathbb{Z}^+$, let $K_{n_1, n_2, \ldots, n_k}$ denote the complete multipartite graph with independent sets of size $n_1, n_2, \ldots, n_k$. Every complete multipartite graph has chordal complement. Furthermore, as observed in the literature (see \cite{MaderEdgeBound} and  \cite{hadwigerSurvey}), $K_{2,2,2,2,2}$ has $h(K_{2,2,2,2,2}) = 7$, yet $|E(K_{2,2,2,2,2})| > 6|V(K_{2,2,2,2,2})|-\binom{6+1}{2}$.

\section{Definitions and Preliminaries}
In this section we begin by briefly introducing our notation. Then we state the definition and some basic facts on the Colin de Verdi\`{e}re parameter, prove the observations from the introduction, and prove two lemmas that will be used in both of the next sections. In Section 3 we prove our main theorem, Theorem \ref{mainThm}, for chordal graphs and the complement of chordal graphs. Finally, in Section 4 we prove Theorem \ref{mainThm} for graphs $G$ with $\mu(G) \leq 7$ or $\mu(G) \geq |V(G)|-6$.

Let $G$ be a graph. We will write an edge connecting vertices $u$ and $v$ as $uv$. We write $\delta(G)$ for the minimum degree, $\Delta(G)$ for the maximum degree, and $\omega(G)$ for the clique number of $G$. The set of vertices adjacent to a vertex $v$ is denoted $N(v)$. The degree of a vertex $v$ in $G$ is written $d_G(v)$, or simply $d(v)$ if the graph is understood from context. For $S \subseteq V(G)$, we write $G[S]$ for the induced subgraph of $G$ with vertex set $S$, and $G-S$ for the induced subgraph of $G$ with vertex set $V(G)-V(S)$. For a vertex $v$ we will write $G-v$ for $G - \{v\}$. If $e$ is an edge of $G$, we write $G/e$ for the graph obtained from $G$ by contracting $e$ and deleting all parallel edges. We will use $A \coloneqq B$ to mean that $A$ is defined to be $B$.

Next we give the definition of the Colin de Verdi\`{e}re parameter. Let $n$ be the number of vertices of $G$. It will be convenient to assume that $V(G) = \{1,2,\ldots, n\}$ and that $G$ is connected. If $G$ is not connected, then define $\mu(G)$ to be the maximum among all connected components $H$ of $G$ of $\mu(H)$. We denote $I \coloneqq \{ii: i \in \{1,2,\ldots, n\}\}$.

\begin{definition}
The Colin de Verdi\`{e}re parameter $\mu(G)$ is the maximum corank of any real, symmetric $n \times n$ matrix $M$ such that:
\begin{enumerate}
\item $M_{i,j} = 0$ if $ij \notin E(G)\cup I$, and $M_{ij} < 0$ if $ij \in E(G)$.
\item $M$ has exactly one negative eigenvalue.
\item If $X$ is a symmetric $n \times n$ matrix such that $MX=0$ and $X_{ij} = 0$ for $ij \in E\cup I$, then $X=0$.
\end{enumerate}
\end{definition}

From the survey of van der Holst, Lov\'{a}sz, and Schrijver, we have:
\begin{theorem}
\label{basicCDV}
\cite{survey} Let $G$ be a graph, let $H$ be a minor of $G$, and let $v \in V(G)$. Then
\begin{enumerate}[(i)]
\itemsep0em
\item $\mu(H) \leq \mu(G)$
\item For every positive integer $t$, $\mu(K_t)=t-1$. 
\item $\mu(G) \leq \mu(G-v)+1$. If $N(v) = V(G)-\{v\}$ and $E(G) \neq \emptyset$ then $\mu(G) = \mu(G-v)+1$. 
\end{enumerate}
\end{theorem}

Then Observation \ref{Obs:tight}, which we restate below, follows from induction on $t$ by (iii) above and noting that for any positive integers $t \geq 3$ and $n$, $(t-1)(n-1)-\binom{t}{2}+n-1 = tn-\binom{t+1}{2}$.

\newtheorem*{Obs:tight}{Observation \ref{Obs:tight}}
\begin{Obs:tight}
Let $H$ be any edge-maximal planar graph on at least 4 vertices and let $t \geq 3$ be an integer. Let $G$ denote the join of $H$ and $K_{t-3}$. Then $\mu(G)=t$ and $|E(G)| = t|V(G)|-\binom{t+1}{2}$.
\end{Obs:tight}

To relate the extremal problem to the graph complement conjecture for Colin de Verdi\`{e}re parameter, and for the next two sections, it will be convenient to state the following lemma.

\begin{lemma}
\label{complement}
Let $G$ be a graph on $n$ vertices and let $t$ be an integer with $n \geq t$. Then $|E(G)| \leq tn-\binom{t+1}{2}$ if and only if $|E(\overline{G})| \geq \binom{n-t}{2}$.
\end{lemma}
\begin{proof}
Observe that $\binom{n-t}{2}+tn-\binom{t+1}{2} = \binom{n}{2} = |E(G)|+|E(\overline{G})|$.
\end{proof}

We will also need the following theorem of Pendavingh (Theorem 5).

\begin{theorem} \cite{Pendavingh}
\label{pendavingh}
If $G$ is a connected graph, then either $|E(G)| \geq \binom{\mu(G)+1}{2}$ or $|E(G)| \geq \binom{\mu{G}+1}{2}-1$ and $G$ is isomorphic to $K_{3,3}$.
\end{theorem}

Now we are ready to prove:

\newtheorem*{Obs:GCCimplies}{Observation \ref{Obs:GCCimplies}}
\begin{Obs:GCCimplies}
If for every graph $G$, $\mu(G)+\mu(\overline{G}) \geq |V(G)|-2$, then every graph $G$ has $|E(G)| \leq (\mu(G)+1)|V(G)| - \binom{\mu(G)+2}{2}$.
\end{Obs:GCCimplies}

\begin{proof}
Let $G$ be a graph on $n$ vertices. Since $\mu(\overline{G})$ is the maximum Colin de Verdi\`{e}re parameter of any connected component of $G$, by Theorem \ref{pendavingh} either $\overline{G}$ is isomorphic to the disjoint union of $K_{3,3}$ and an independent set of vertices, or $|E(\overline{G})| \geq \binom{\mu(\overline{G})+1}{2}$. In the latter case, $|E(\overline{G})| \geq \binom{\mu(\overline{G})+1}{2} \geq \binom{n-1-\mu(G)}{2}$. So by Lemma \ref{complement}, we are done.

If $\overline{G}$ is isomorphic to the disjoint union of $K_{3,3}$ and a set of $k$ independent vertices, then $\mu(\overline{G}) = 3$ and by (iii) of Theorem \ref{basicCDV} and since $\mu(\overline{K_{3,3}}) = 2$, $\mu(G) = k+2$. So then $\mu(G)+\mu(\overline{G}) = n-1$. So $$|E(\overline{G})| \geq \binom{\mu(\overline{G})+1}{2}-1 = \binom{n-\mu(G)}{2}-1 \geq \binom{n-1-\mu(G)}{2}$$
and again we are done by Lemma \ref{complement}.
\end{proof}

We finish this section by proving some basic facts about a counterexample to the main Conjecture \ref{mainConjecture} such that every induced subgraph on one less vertex satisfies the conjecture. This lemma will be used in Sections 3 and 4 to help prove our main Theorem \ref{mainThm}.

\begin{lemma}
\label{basicCE}
Let $G$ be an $n$-vertex graph with $|E(G)| > \mu(G)n-\binom{\mu(G)+1}{2}$. Suppose also that for every $x \in V(G)$, $|E(G-x)|\leq \mu(G-x)(n-1)-\binom{\mu(G-x)+1}{2}$. Then $\mu(G) < \delta(G)\leq \Delta(G) < n-1$.
\end{lemma}

\begin{proof}
Suppose $v$ is a vertex of $G$ with $d(v) \leq \mu(G)$. Then by Theorem \ref{basicCDV}, we have $\mu(G-v) \in \{\mu(G), \mu(G)-1\}$ and $\mu(G) \leq n-1$. Then
\begin{multline*}|E(G)| =|E(G-v)|+d(v) \leq \mu(G-v)(n-1)-\binom{\mu(G-v)+1}{2}+\mu(G) \\
\leq \mu(G)(n-1)- \binom{\mu(G)+1}{2}+\mu(G) = \mu(G)n-\binom{\mu(G)+1}{2}
\end{multline*}
a contradiction.

If $u$ is a vertex with $d(u) = n-1$, first note that $E(G) \neq \emptyset$. Then by (iii) of Theorem \ref{basicCDV}, $\mu(G-u) = \mu(G)-1$, and so 
$$|E(G)| = |E(G-u)| + n-1\leq (\mu(G)-1)(n-1)-\binom{\mu(G)}{2}+n-1 = \mu(G)n-\binom{\mu(G)+1}{2}$$
a contradiction.
\end{proof}

\section{Chordal Graphs and Complements of Chordal Graphs}
In this section we will show that if $G$ is a graph such that $G$ is chordal or $\overline{G}$ is chordal, then $|E(G)| \leq \mu(G)|V(G)|-\binom{\mu(G)+1}{2}$. Define a \textit{simplicial} vertex of a graph $G$ to be a vertex $v$ such that $G[N(v)]$ is a complete graph. We will use the fact that every chordal graph has a simplicial vertex.

\begin{lemma}
If $G$ is a chordal graph then $|E(G)| \leq \mu(G)|V(G)| - \binom{\mu(G)+1}{2}$.
\end{lemma}
\begin{proof}
Let $G$ be a vertex-minimal counterexample. Let $u$ be a simplicial vertex of $G$. Then $d(u) \leq \omega(G)-1 \leq \mu(G)$. This is a contradiction to Lemma \ref{basicCE} since every induced subgraph of a chordal graph is chordal and $G$ is a vertex-minimal counterexample.
\end{proof}

For graphs with chordal complement, we need to introduce the following two theorems. Mitchell and Yengulalp showed that:

\begin{theorem}
\cite{chordalGraphs}
\label{chordalBound}
If $G$ is a chordal graph, then $\mu(G)+\mu(\overline{G}) \geq |V(G)|-2$.
\end{theorem}

For an integer $t \geq 3$, let $K_t-\Delta$ denote the graph obtained from $K_t$ by deleting the edges of a triangle. Fallat and Mitchell proved that:
\begin{theorem}
\cite{chordalChar}
\label{chordalChar}
Let $G$ be a chordal graph. Then $\mu(G) = \omega(G)$ if and only if $G$ has $K_{\omega(G)+2}-\Delta$ as an induced subgraph. Otherwise $\mu(G) = \omega(G)-1$.
\end{theorem}

We are now ready to prove the final lemma of this section.

\begin{lemma}
If $G$ is a graph so that $\overline{G}$ is chordal, then $|E(G)| \leq \mu(G)|V(G)|-\binom{\mu(G)+1}{2}$.
\end{lemma}

\begin{proof}
Let $G$ be a vertex-minimal counterexample, and set $n \coloneqq |V(G)|$. First we show two claims:

\begin{claim}
$\omega(\overline{G}) \geq 2$
\end{claim}
\begin{proof}
Otherwise $G$ is a complete graph and by (ii) of Theorem \ref{basicCDV}, $\mu(G) = n-1$. Then $|E(G)| = \binom{n}{2} = \mu(G)n-\binom{\mu(G)+1}{2}$, a contradiction. 
\end{proof}

\begin{claim}
$\Delta(\overline{G}) < \mu(\overline{G})+1$
\end{claim}
\begin{proof}
Otherwise by Theorem \ref{chordalBound}, $\Delta(\overline{G}) \geq \mu(\overline{G})+1 \geq n-1-\mu(G)$. Then $\delta(G) \leq \mu(G)$, a contradiction to Lemma \ref{basicCE} since $G$ is a vertex-minimal counterexample and every induced subgraph of $G$ has chordal complement.
\end{proof}

Now, suppose $\mu(\overline{G}) = \omega(\overline{G})$. Then by Theorem \ref{chordalChar}, $\overline{G}$ has an induced subgraph that is isomorphic to $K_{\omega(\overline{G})+2}-\Delta$. Since $\omega(\overline{G})\geq 2$, we have $\Delta(\overline{G}) \geq \Delta(K_{\omega(\overline{G})+2}-\Delta) = \omega(\overline{G})+1 = \mu(\overline{G})+1$, a contradiction to Claim 2.

So $\mu(\overline{G}) = \omega(G)-1$. Let $S \subseteq V(\overline{G})$ be the set of vertices of a maximum clique of $\overline{G}$. Write $\overline{S} \coloneqq V(\overline{G})-S$. First we will show that if $x \in S$ and $y \in \overline{S}$, then $xy \notin E(\overline{G})$. If $xy \in E(\overline{G})$, then $d_{\overline{G}}(x) \geq \omega(\overline{G}) = \mu(\overline{G})+1$, a contradiction to Claim 2.

If $\overline{S} = \emptyset$, then $E(G) = \emptyset$ and $G$ would satisfy the lemma. So $\overline{S} \neq \emptyset$. Then let $u \in S$ and $v \in \overline{S}$. We have $uv \in E(G)$. Let $uv$ also denote the new vertex of $G/uv$. Since in $\overline{G}$ the vertex $u$ is adjacent to no vertices in $\overline{S}$ and $v$ is adjacent to no vertices in $S$, the vertex $uv$ is adjacent to every other vertex in $G/uv$. Also, since $|S| \geq 2$, $G/uv$ contains an edge. So by (iii) of Theorem \ref{basicCDV}, $\mu(G/uv) = \mu(G-\{u,v\})+1$. Then $|E(G-\{u,v\})| \leq (\mu(G)-1)(n-2)-\binom{\mu(G)}{2}$.

Also, $d_{\overline{G}}(u) = \omega(\overline{G})-1$, so $d_G(u) = n-\omega(\overline{G}) = n-1-\mu(\overline{G}) \leq \mu(G)+1$ by Theorem \ref{chordalBound}. By Lemma \ref{basicCE}, $d_G(y) < n-1$. Then
\begin{multline*}
|E(G)| = |E(G-\{u,v\})|+d_G(u)+d_G(v)-1 \leq (\mu(G)-1)(n-2)-\binom{\mu(G)}{2}+\mu(G)+n-2\\
= \mu(G)n-\binom{\mu(G)+1}{2}
\end{multline*}
a contradiction.
\end{proof}

\section{Graphs with Small or Large Parameter}
In this section we will show that graphs $G$ such that either $\mu(G) \leq 7$ or $\mu(G) \geq |V(G)|-6$ have $|E(G)| \leq \mu(G)|V(G)|-\binom{\mu(G)+1}{2}$. First we give some definitions related to clique sums.

Let $k$ be a non-negative integer and let $G_1$ and $G_2$ be two vertex-disjoint graphs. For $i = 1,2$ let $C_i \subseteq V(G_i)$ be a clique of size $k$ of $G_i$. Then let $G$ denote the graph obtained from $G_1$ and $G_2$ by identifying the vertices in cliques $C_1$ and $C_2$ by some bijection. We say $G$ is a \textit{pure} $k$-clique sum of $G_1$ and $G_2$.

Let $H$ be some fixed graph and let $k$ be a non-negative integer. We say a graph $G$ is \textit{built} by pure $k$-sums of $H$ if either $G$ is isomorphic to $H$, or if $G$ is a pure $k$-clique sum of graphs $H_1$ and $H_2$, where $H_1$ and $H_2$ are built by pure $k$-sums of $H$. The following generalization of Theorem \ref{hadwigerEdgeBound} is due to J{\o}rgensen.

\begin{theorem}
\cite{JorgensenEdgeBound}
\label{K8}
Let $G$ be a graph with $h(G) \leq 7$, $|V(G)| \geq 6$, and $|E(G)| > 6|V(G)|-21$. Then $|E(G)| = 6|V(G)|-20$, and $G$ can be built by pure 5-sums of $K_{2,2,2,2,2}$.
\end{theorem}

For graphs with no $K_9$ minor, Song and Thomas proved:

\begin{theorem}
\cite{SongThomas}
\label{K9}
Let $G$ be a graph with $h(G) \leq 8$, $|V(G)| \geq 7$, and $|E(G)| > 7|V(G)|-28$. Then $|E(G)| = 7|V(G)|-27$, and either either $G$ is isomorphic to $K_{2,2,2,3,3}$, or $G$ can be built by pure 6-sums of $K_{1,2,2,2,2,2}$.
\end{theorem}

We will also make use of the following theorem due to Kotlov, Lov\'{a}sz, and Vempala.

\begin{theorem}
\cite{sphereRep}
\label{largeMu}
If $G$ is a graph with $\mu(G) \leq 3$, then $\mu(G) + \mu(\overline{G})\geq |V(G)|-2$. 
\end{theorem}

Kotlov, Lov\'{a}sz, and Vempala also characterized exactly which graphs $G$ have $\mu(G) \geq |V(G)|-3$ (Theorems 3.3 and 5.2, \cite{sphereRep}). Let $P_{3,2}$ denote the graph formed from three disjoint paths of length two by identifying one end from each path. That is, $P_{3,2}$ is the graph in Figure \ref{P}. We will make use of the following corollary of these theorems:

\begin{figure}
\centering
\tikzstyle{every node}=[circle, draw, fill=black!50,
                        inner sep=0pt, minimum width=6pt]
\begin{tikzpicture}{
\node (x0) at (2,2) {};
\node (1) at (1.2,1) {};
\node (2) at (2,1) {};
\node (3) at (2.8,1) {};
\node (x1) at (0.6,0) {};
\node (x2) at (2,0) {};
\node (x3) at (3.4,0) {};
\foreach \i in {1,2,3}{
\draw (x0) -- (\i);
\draw (\i) -- (x\i);
}
};
\end{tikzpicture}
\caption{The graph $P_{3,2}$}
\label{P}
\end{figure}
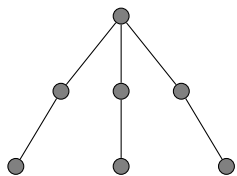

\begin{corollary}
\label{noSubgraph}
\cite{sphereRep}
If $G$ is a graph such that $\overline{G}$ contains no $P_{3,2}$ subgraph and no cycle, then $\mu(G) \geq |V(G)|-3$.
\end{corollary}

Now we are ready to prove the following lemma.

\begin{lemma}
\label{smallMu}
Let $G$ be a graph with $\mu(G) \leq 7$. Then $|E(G)| \leq \mu(G)|V(G)|-\binom{\mu(G)+1}{2}$.
\end{lemma}
\begin{proof}
First note that $\mu(K_{2,2,2,2,2}) \geq 7$, $\mu(K_{1,2,2,2,2,2}) \geq 8$, and $\mu(K_{2,2,2,3,3}) \geq 8$ by Theorem \ref{largeMu}, since $\mu(\overline{K_{2,2,2,2,2}}) =1$, $\mu(\overline{K_{1,2,2,2,2,2}}) =1$, and $\mu(\overline{K_{2,2,2,3,3}})=2$.

Let $G$ be a graph with $\mu(G) \leq 7$, and write $n \coloneqq |V(G)|$. If $\mu(G) \leq 5$, then since $h(G) \leq \mu(G)+1$, the lemma follows from Theorem \ref{hadwigerEdgeBound}. If $\mu(G) = 6$, then $G$ does not contain $K_{2,2,2,2,2}$ as a subgraph. So we are done by Theorem \ref{K8}. If $\mu(G)=7$, then $G$ does not contain $K_{1,2,2,2,2,2}$ or $K_{2,2,2,3,3}$ as a subgraph, and we are done by Theorem \ref{K9}. 
\end{proof}

For the next lemma we need to give some definitions related to subdivisions. Fix a graph $H'$. We say a graph $H$ is a \textit{subdivision} of $H'$ if $H$ can be formed from $H'$ by replacing edges of $H'$ with internally-disjoint paths with the same ends. Then we say $v \in V(H)$ is a \textit{branch} vertex of $H$ if also $v \in V(H')$. Suppose $H'$ is a bipartite graph with bipartition $(A,B)$. That is, $(A,B)$ is a partition of the vertex set of $H'$ such that every edge of $H'$ has one end in $A$ and one end in $B$. Then if $H$ is a subdivision of $H'$, we will say that branch vertices $u$ and $v$ of $H$ are \textit{in the same part} of $H$ if either $u,v \in A$ or $u,v \in B$. Now we are ready to prove the final lemma:

\begin{lemma}
Let $G$ be an $n$-vertex graph with $\mu(G) \geq n-6$. Then $|E(G)| \leq \mu(G)n-\binom{\mu(G)+1}{2}$.
\end{lemma}
\begin{proof}
Let $G$ be a vertex-minimal counterexample. Write $n \coloneqq |V(G)|$ and $c \coloneqq n - \mu(G)$. First we will show that $\delta(\overline{G})\geq 1$. Let $v\in V(G)$. Then by part (iii) of Theorem \ref{basicCDV}, $\mu(G-v) \geq \mu(G)-1 \geq |V(G-v)|-6$. So by Lemma \ref{basicCE}, $\Delta(G) < n-1$. So $\delta(\overline{G})\geq 1$.

Next we find upper and lower bounds for $n$. By Lemma \ref{smallMu}, we may assume $\mu(G)\geq 8$, so $n = \mu(G)+c \geq 8+c$. By Lemma \ref{complement}, $|E(\overline{G})| < \binom{n-\mu(G)}{2} = \binom{c}{2}$. Then since $\delta(\overline{G})\geq 1$, we have $n \leq 2|E(\overline{G})|\leq 2(\binom{c}{2}-1)$. In total, we have $8+c \leq n \leq 2(\binom{c}{2}-1)$. This implies that $c \geq 5$. 

Now we will show that $\mu(\overline{G}) \geq c-2$. Otherwise, $\mu(\overline{G}) \leq c-3\leq 3$. Then by Theorem \ref{largeMu}, $n-2 \leq \mu(G)+\mu(\overline{G}) \leq n-3$, a contradiction. Now we proceed by cases.

\begin{flushleft}
\textit{Case 1:} $c = 5$.
\end{flushleft}
Then since $\mu(\overline{G}) \geq c-2 =3$, $\overline{G}$ is not outerplanar. So $\overline{G}$ has a subgraph $H$ that is either a subdivision of $K_4$ or a subdivision of $K_{2,3}$. Let $D \subseteq V(H)$ be the set of branch vertices of $H$. Then since $\delta(\overline{G}) \geq 1$ and $n \geq 8+c=13$,
$$\binom{5}{2} > |E(\overline{G})| \geq \frac{1}{2}\left(\sum_{x \in D}d_H(x)+\sum_{y \in V(G)-D}d_{\overline{G}}(y)\right) \geq \frac{1}{2}\left(\sum_{x \in D}d_H(x)+13-|D|\right)$$
In either case we get a contradiction.

\begin{flushleft}
\textit{Case 2:} $c = 6$.
\end{flushleft}
Then $\mu(\overline{G})\geq 4$ and so $\overline{G}$ is not planar. So $\overline{G}$ has a subgraph $H$ that is either a subdivision of $K_5$ or a subdivision of $K_{3,3}$. If $H$ is a subdivision of $K_5$ then similarly to before, since $\delta(\overline{G}) \geq 1$ and $n \geq 8+c=14$, we have $\binom{6}{2}-1 \geq |E(\overline{G})| \geq \frac{1}{2}(5*4+9) = \frac{29}{2}$, a contradiction.

So $H$ is a subdivision of $K_{3,3}$. Let $u,v \in V(H)$ be distinct branch vertices of $H$ that are in the same part of $H$ such that $d_{\overline{G}}(u)+d_{\overline{G}}(v)$ is maximum. We will show that $\overline{G-\{u,v\}}$ contains no $P_{3,2}$ subgraph and no cycle. Write $k \coloneqq d_{\overline{G}}(u)+d_{\overline{G}}(v)-6$. Then $k \geq 0$. Since $u$ and $v$ are not adjacent in $H$ and vertices adjacent to $u$ or $v$ in $H$ have degree at least 1 in $\overline{G-\{u,v\}}$, the graph $\overline{G-\{u,v\}}$ has at most $k$ vertices of degree 0. 

Suppose $\overline{G-\{u,v\}}$ has a $P_{3,2}$ subgraph. If $u$ and $v$ are not adjacent in $\overline{G-\{u,v\}}$, then
$$\binom{6}{2}-1 \geq |E(\overline{G})| = k+6+|E(\overline{G-\{u,v}\})|\geq k+6+|E(P_{3,2})|+\frac{1}{2}(12-|V(P_{3,2})|-k)\geq \frac{29}{2}$$
a contradiction. If $u$ and $v$ are adjacent in $\overline{G-\{u,v\}}$, then since they are not adjacent in $H$, we have $k \geq 2$. So similarly we have
$$\binom{6}{2}-1 \geq |E(\overline{G})| = k+5+|E(\overline{G-\{u,v}\})|\geq k+5+|E(P_{3,2})|+\frac{1}{2}(12-|V(P_{3,2})|-k)\geq \frac{29}{2}$$
again a contradiction. So $G$ contains no $P_{3,2}$ subgraph. 

Now we will show that $\overline{G-\{u,v\}}$ has no cycle. Write $S \coloneqq V(G) - V(H)$. Let $S_1$ be the set of vertices in $S$ with degree strictly greater than 1 in $\overline{G}$. Write $d \coloneqq \sum_{z \in V(H)}d_{\overline{G}}(z)-d_H(z)$. Then since $\delta(\overline{G}) \geq 1$ and $n \geq 14$, we have:
\begin{multline*}
\binom{6}{2}-1 \geq |E(\overline{G})| = \frac{1}{2}\left(\sum_{x \in V(H)}d_H(x)+d+\sum_{y \in S}d_{\overline{G}}(y)  \right) \\
\geq \frac{1}{2}\left(\sum_{x \in V(H)}d_H(x)+d+14-|V(H)|+|S_1|  \right) = \frac{1}{2}\left( |V(H)|+d+|S_1|+20 \right)
\end{multline*}
So $|V(H)|+d+|S_1|\leq 8$. Since $|V(H)| \geq 6$, we have that $d+|S_1| \leq 2$.

Suppose $\overline{G-\{u,v\}}$ contains a cycle $C$. If $|V(C)\cap S| \geq 3$, then $|S_1| \geq |V(C)\cap S| \geq 3$, a contradiction. If $|V(C)\cap S| \in  \{1,2\}$, then $\overline{G-\{u,v\}}$ has at least two edges with one end in $V(H)$ and the other in $S$. Then $d \geq 2$, and $|S_1| \geq |V(C)\cap S|\geq 1$, a contradiction. 

Finally, suppose $|V(C) \cap S| = 0$. Since $u$ and $v$ are in the same part of $H$, the graph $H-\{u,v\}$ contains no cycle. So there exist distinct vertices $a,b \in V(H)-\{u,v\}$ so that $a$ and $b$ are adjacent in $\overline{G - \{u,v\}}$ but not in $H$. Then we have $d_{\overline{G}}(a)-d_H(a),d_{\overline{G}}(b)-d_H(b)>0$, so $d \geq 2$. Then since $|V(H)|+d\leq 8$, we have $|V(H)|=6$ and $H$ is isomorphic to $K_{3,3}$. Then since $a$ and $b$ are not adjacent in $H$, they are in the same part of $H$. So by the choice of $u$ and $v$, we have $d_{\overline{G}}(u)+d_{\overline{G}}(v) \geq d_{\overline{G}}(a)+d_{\overline{G}}(b)\geq 8$. Then $d \geq d_{\overline{G}}(u)+d_{\overline{G}}(v) -6+d_{\overline{G}}(a)+d_{\overline{G}}(b)-6 \geq 4$, a contradiction.

We have shown that $\overline{G-\{u,v\}}$ has no cycle and no $P_{3,2}$ subgraph. Then by Corollary \ref{noSubgraph}, we have $n-6=\mu(G) \geq \mu(G-\{u,v\}) \geq |V(G-\{u,v\})|-3 = n-5$, a contradiction. This completes the proof.
\end{proof}

\section*{Acknowledgments}
I would like to thank Robin Thomas for his careful proof-reading of this manuscript and for helpful suggestions that lead to Observations 1 and 2.

\let\OLDthebibliography\thebibliography
\renewcommand\thebibliography[1]{
  \OLDthebibliography{#1}
  \setlength{\parskip}{0em}
  \setlength{\itemsep}{0pt plus 0.3ex}
}

\vspace*{\fill}
This material is based upon work supported by the National Science Foundation. Any opinions, findings, and conclusions or recommendations expressed in this material are those of the author and do not necessarily reflect the views of the National Science Foundation.
\end{document}